\documentclass[a4paper,reqno,12pt]{amsart}
\usepackage[english]{babel}

\usepackage[utf8]{inputenc}
\usepackage[T1]{fontenc}
\usepackage{amsfonts,amsmath,amssymb,amsthm}

\usepackage[expansion=false]{microtype}
\usepackage{mathtools}
\usepackage{geometry}
\usepackage{textcomp}
\usepackage{url}
\usepackage[usenames]{color}
\usepackage{faktor}
\usepackage{xfrac}
\usepackage{euscript}
\usepackage{tikz}
\usepackage{tikz-cd} 
\usepackage{enumitem}
\usepackage{graphicx}

\theoremstyle{theorem}
\newtheorem{theorem}{Theorem}[section]

\theoremstyle{definition}
\newtheorem{definition}[theorem]{Definition}

\newtheorem{conjecture}[theorem]{Conjecture}

\newtheorem{construction}[theorem]{Construction}
\newtheorem{example}[theorem]{Example}

\theoremstyle{theorem}
\newtheorem{lemma}[theorem]{Lemma}

\theoremstyle{remark}
\newtheorem{remark}[theorem]{Remark}

\DeclareMathOperator{\Spec}{Spec}
\DeclareMathOperator{\GL}{GL}

\DeclareMathOperator{\Aut}{Aut}

\DeclareMathOperator{\A}{\mathbb A}
\DeclareMathOperator{\TT}{\mathbb T}

\DeclareMathOperator{\ZZ}{\mathbb Z}

\newcommand{\lmt}{\longmapsto}

\newcommand{\UU}{\mathcal U}
\newcommand{\kk}{\Bbbk}

\newcommand{\0}{\circ}

\makeatletter
\DeclareRobustCommand\widecheck[1]{{\mathpalette\@widecheck{#1}}}
\def\@widecheck#1#2{%
    \setbox\z@\hbox{\m@th$#1#2$}%
    \setbox\tw@\hbox{\m@th$#1%
       \widehat{%
          \vrule\@width\z@\@height\ht\z@
          \vrule\@height\z@\@width\wd\z@}$}%
    \dp\tw@-\ht\z@
    \@tempdima\ht\z@ \advance\@tempdima2\ht\tw@ \divide\@tempdima\thr@@
    \setbox\tw@\hbox{%
       \raise\@tempdima\hbox{\scalebox{1}[-1]{\lower\@tempdima\box
\tw@}}}%
    {\ooalign{\box\tw@ \cr \box\z@}}}
\makeatother

\title{Automorphism groups of semi-rigid $T$-varieties of complexity one}
\author{Kirill Rassolov}
\address{%
    \begin{minipage}{\textwidth}
        \setlength{\parindent}{0pt}
        \textbf{E-mail:} \url{kirill.rassolov@math.msu.ru}\\[4pt]
        Lomonosov Moscow State University, Faculty of Mechanics and Mathematics,\\
        Department of Higher Algebra, Leninskie Gory 1, Moscow, 119991 Russia;\\[2pt]
        and\\[2pt]
        HSE University, Faculty of Computer Science,\\
        Pokrovsky Boulevard 11, Moscow, 109028 Russia
    \end{minipage}%
}
\date{}

\thanks{The work is supported by the grant RSF 25-11-00302}
\begin{document}
\maketitle

\begin{abstract}
    We prove that if $X$ is a semi-rigid normal rational affine algebraic variety with torus action of complexity one, with only constant invertible global functions and finitely generated divisor class group, then the identity component of the automorphism group of $X$ is a semi-direct product of a torus and the maximal unipotent-generated subgroup.
\end{abstract}




\section*{Introduction}
Let $\kk$ be an algebraically closed field of characteristic zero and $\kk^+$ denote its additive group. It is well-known that for a projective variety $X$ over $\kk$, the automorphism group $\Aut(X)$ is locally algebraic. Although an affine variety $X$ may not satisfy this property, $\Aut(X)$ carries a natural structure of so-called ing-group, or infinite-dimensional algebraic group. This structure is defined such that a regular action of an algebraic group $G$ on $X$ is the same an a closed embedding of $G$ into the ind-group $\Aut(X)$. We refer to~\cite{FK} for a systematic treatment. 
 
The question of whether the connected component $\Aut(X)^\circ$ is an algebraic group is related to study of actions of the additive group $\kk^+$. Recall that an affine variety $X$ is called rigid if $X$ admits no nontrivial $\kk^+$-action. The following conjecture can be found in explicit form in~\cite{PZ}; see also~\cite[Question~9.1.5]{FK}.

\begin{conjecture}\label{RigidConj}
    If $X$ is a rigid affine variety, then $\Aut(X)^\circ$ is an algebraic torus.
\end{conjecture}
Let us note that if a variety $X$ is neither rigid nor isomorphic to $\A^1$, then $\Aut(X)^\circ$ is infinite-dimensional~\cite[Proposition 7.5]{K}. Thus the conjecture means that the torus is the only algebraic group that appears as $\Aut(X)^\circ$ for affine $X$, except for the case $X=\A^1$ and $\Aut(X)=\kk^\times\ltimes\kk^+$. Conjecture~\ref{RigidConj} has been addressed in many papers and verified in some particular cases including surfaces~\cite{PZ, BeP}, toric varieties~\cite{BoldG} (cf. also Example~2.3 and Remark~2.4 in~\cite{AG}) and $T$-varieties of complexity one~\cite{BorG}.

Developing this area, Perepechko and Regeta proposed a generalization of Conjecture~\ref{RigidConj} concerning the varieties with few $\kk^+$-actions. We say that $X$ is semi-rigid if all $\kk^+$-actions on $X$ share the same ring of invariants $\mathcal O(X)^{\kk^+}$. This is the same as all $\kk^+$-actions on $X$ commute, and if that holds, all $\kk^+$-actions are 'rational replicas' of each other, cf. Lemma~2.2 in~\cite{PR23}. Finally, let $\UU(X)$ stand for the subgroup generated by all $\kk^+$-subgroups in the automorphism group.

\begin{conjecture}(\cite{PR23}, cf. also \cite{PR24})\label{Semi-rigidConj} If $X$ is a semi-rigid affine variety, then $\Aut(X)^\circ=T\ltimes \UU(X)$ where $T$ is a maximal torus in $\Aut(X)$.
    
\end{conjecture}

Conjecture~\ref{Semi-rigidConj} is known to hold if the variety $X$ is toric, see~\cite{BoldG}. 

We denote by $\Aut_{alg}(X)$ the subgroup of $ \Aut(X)^\circ$ generated by all algebraic elements. Recall that an element $g$ is called \emph{algebraic} if it is contained in a connected algebraic subgroup. Thus $\Aut_{alg}(X)$ is generated by all connected algebraic subgroups in the automorphism group. 

It is proved in~\cite{PR24} that the following conditions are equivalent:
\begin{enumerate}
    \item[(i)] the group $\Aut(X)^\0$ is generated by algebraic elements;
    \item[(ii)] each element of $\Aut(X)^\circ$ is algebraic;
    \item[(iii)] the group $\Aut(X)^\0$ 
    is nested, i.e., is the direct limit of algebraic groups;
    \item[(iv)] $\Aut(X)^\0 = T\ltimes \UU(X)$.
\end{enumerate}
Moreover, each of these conditions \emph{implies} any of the following (and these new conditions are again equivalent, see~\cite{PR23}):
\begin{enumerate}
    \item[$\bullet$] each element of $\Aut_{alg}(X)$ is algebraic;
    \item[$\bullet$] the subgroup $\Aut_{alg}(X)\subseteq \Aut(X)^\circ$ is nested;
    \item[$\bullet$] $\Aut_{alg}(X) = T\ltimes \UU(X)$;
    \item[$\bullet$] $\UU(X)$ is abelian;
    \item[$\bullet$] $X$ is semi-rigid.
\end{enumerate}

This note addresses Conjecture~\ref{Semi-rigidConj} and verifies it in the following particular case:

\begin{theorem}\label{Theorem}
    Let $X$ be an affine semi-rigid variety and $\Aut(X)$ contain a maximal torus $T$ of dimension $\dim T = \dim X-1$. Suppose $X$ is normal, rational, with only constant invertible global functions and finitely generated divisor class group. Then $\Aut(X)^\circ=T\ltimes\UU(X)$. 
\end{theorem}

\section*{Acknowledgements} The author is grateful to Sergey Gaifullin for posing the problem and useful discussions of this note. Thanks also go to Alexander Perepechko for pinpointing some inaccuracies in the first version of the manuscript.

\section{Trinomial varieties}
In this section we gather some definitions on trinomial varieties which are useful in the proof of Theorem~\ref{Theorem}. 
We define trinomial varieties  following~\cite{HW}.

\begin{construction}~\cite[Construction~1.1]{HW}
 Fix integers $n,r\ge2$ and consider a partition $n=n_{q}+\ldots+n_r$ into positive integers $n_i$ with $q\in\{0,1\}$. Let $T_{ij}$, $i=q,\ldots,r$ and $j=1,\ldots,n_i$, be independent variables. Choose also an integer $m\ge0$ and consider additional variables $S_k$ with $k=1,\ldots, m$. We will view the $T_{ij}, S_k$ as coordinates in the affine space $\A^{n+m}$ and abbreviate the corresponding polynomial algebra in $n+m$ variables by $\kk[T_{ij}, S_k]$. For each $i=q,\ldots, r$ choose a tuple $l_i=(l_{i1},\ldots, l_{in_i})$ of positive integers $l_{ij}$ and consider the monomial
 $$
 T_{i}^{l_i} = T_{i1}^{l_{i1}}\cdots T_{in_i}^{l_{in_i}}\in \kk[T_{ij}, S_k].
 $$ 

We now define the trinomial algebra $R(A)$ depending on some data $A$. We allow two types of these data.

\textit{Type $1$.} Set $q=1$ and let $A=(a_1,\ldots, a_r)$ be a tuple of pairwise distinct elements $a_i\in\kk$. Then $R(A)$ is the quotient of $\kk[T_{ij}, S_k]$ by the following relations:
\begin{equation}\label{equations1}
    T_i^{l_i}-T_{i+1}^{l_{i+1}} - (a_{i+1}-a_i)=0,\quad i=1,\ldots, r-1
\tag{1.1}
\end{equation}

\textit{Type $2$}. Set $q=0$ and let $A$ be a $2\times(r+1)$-matrix 
$$
A = \begin{pmatrix}
    a_{01} &\ldots &a_{r1}\\
    a_{02} &\ldots &a_{r2}
\end{pmatrix}
$$ with pairwise linearly independent columns, $a_{i1},\,a_{i2}\in\kk$. Then $R(A)$ is the quotient of $\kk[T_{ij}, S_k]$ by the relations:
\begin{equation}\label{equations2}
    \det \begin{pmatrix}
    T_i^{l_i} & T_{i+1}^{l_{i+1}} & T_{i+2}^{l_{i+2}}\\
    a_i & a_{i+1} & a_{i+2}
\end{pmatrix}, \quad i=0,\ldots,r-2.
\tag{1.2}
\end{equation}
\end{construction}
\begin{definition}
    We call $X(A)=\Spec R(A)$ a \emph{trinomial variety.}
\end{definition}

Note that the dimension of the trinomial variety $X(A)=\Spec R(A)$ is $n+m-r+1$. The algebra $R(A)$ admits the natural $\ZZ^{n+m-r}$-grading defined as follows. Suppose that each variable $T_{ij}$ has degree $w_{ij}$ and each $S_k$ has degree $v_k$. These degrees with the relations $\deg T_i^{l_i}=0$ for Type 1 and $\deg T_i^{l_i}=\deg T_j^{l_j}$ for Type 2 generate an abelian group with torsion-free part of rank $n+m-r$, see~\cite{HW} for details. This grading defines an actions of the torus $\TT$ of dimension $n+m-r$. The action of $\TT$ on $X(A)$ comes from the action of the standard torus of dimension $n+m$ on the affine space with coordinates $T_{ij}, S_k$.

\begin{definition}
     We call the torus $\TT$ described above the \textit{canonical torus} of the trinomial variety~$X(A)$.
\end{definition}

\begin{remark} (see \cite{HW})\label{RemarkPoint}
    If $X(A)$ is of Type 2, then the action of the canonical torus $\TT$ is pointed, meaning that the rational polyhedral cone denerated by the degrees occurring in $R(A)$ contains no line.
\end{remark}

\section{Proof of Theorem~\ref{Theorem}}  
The following fact seems to be known and follows from a classical argument. We give the proof below in lack of a precise reference.

\begin{lemma}\label{LemmaHumphreys}
    Let $G$ be an ind-group containing a torus $T$ as a closed subgroup. Then $N_G(T)^\circ = C_G(T)^\circ$.
\end{lemma}
\begin{proof}
    Clearly both $N_G(T)$ and $C_G(T)$ are closed subgroups of $G$, therefore carrying natural structures of ind-groups. By definition, the connected component of the ind-group $N_G(T)$ consists of all elements that can be included in a closed irreducible algebraic subvariety $V\subseteq N_G(T)$ containing the neutral element. Define a morphism $V\times T\to T$ via $(g, t)\longmapsto g tg^{-1}$. According to~\cite[Proposition in 16.3]{H}, this morphism is constant in $g$, hence any $g\in V$ commutes with~$T$.
\end{proof}

In the following lemma, we investigate the normalizer $N(\TT)$ of the canonical torus $\TT$ in the automorphism group of a trinomial variety. Recall that an ind-group $G$ is called locally algebraic if $G^\circ$ is an algebraic group.

\begin{lemma}\label{Trinomial-N-Lemma}
    Let $X$ be a trinomial variety, $\TT$ the canonical torus of $X$.
    \begin{enumerate}
        \item[\emph{(i)}] In case of Type 1, $N(\TT)$ is a locally algebraic group having a torus as identity component.
        \item[\emph{(ii)}] In case of Type 2, $N(\TT)$ is a linear algebraic group.
    \end{enumerate} 
\end{lemma}
\begin{proof}
    Let $X$ be of Type $1$. 
    If $n_i>1$ for some $i$, the assersion follows from~\cite[Lemma~6.2 and Corollary~3.8]{BorG}. 
    Thus it remains to treat the case $n_i=1$ for all $i$, i.e., each $T_{i}^{l_i}$ is a power of a variable. Then $X$ is either rigid or isomorphic to the affine space with coordinates $T_{11}, S_1,\ldots, S_m$ (see the rigidity criterion in~\cite{EGS}). In the first case, the assertion follows from~\cite[Proposition~6.3]{BorG}. 
    
    Now let $X\simeq \Spec \kk[T_{11}, S_1,\ldots, S_m]$ with $\TT\simeq (\kk^\times)^m$ acting on $S_1,\ldots, S_m$.  For any $\varphi\in N(\TT)$, the comorphism $\varphi^*$ leaves $\kk[T_{11}]$ invariant as functions of $\TT$-weight zero. Hence $\varphi^*T_{11}=aT_{11},$ $a\in\kk^\times$. By Lemma~\ref{LemmaHumphreys} we may further assume that $\varphi\in C(\TT)$. Thus $\varphi^*$ must preserve the weights of regular functions with respect to $\TT$, so $\varphi^*S_k=S_kf_k(T_{11})$ for all $k$, $f_k$ being polynomials in one variable. Since $\varphi^*$ is invertible, all the $f_k$ are constant. Therefore $N(\TT)^\circ= \kk^\times\times\TT$. 

    Finally, suppose that $X$ is of Type 2. Then the $\TT$-action is pointed (Remark~\ref{RemarkPoint}), hence $N(\TT)$ is a linear algebraic group by~\cite[Proposition~2.2]{AHHL}. 
\end{proof}

\begin{example}
    Let us show that the connected normalizer $N(\TT)^\0$ need not be a torus. Consider the affine plane $\A^2$ as the trinomial variety of Type~2 given by a linear equation in $\A^3$. The one-dimensional torus $\TT$ acts via $x\lmt tx$, $y\lmt ty$. Note that for any $\varphi\in N(\TT)$, $\varphi^*$ must take $x$ and $y$ to homogeneous polynomials of degree one (since $\varphi^*$ induces an automorphism of the weight latice of $\TT$, and there is no polynomial of negative $\TT$-weight).   
   Hence $N(\TT)$ acts linearly on $\A^2$. It follows that $N(\TT)=C(\TT)=\GL_2$.
\end{example}

\begin{proof}[Proof of Theorem~\ref{Theorem}]
    
    By \cite[Theorem 1.3 and Lemma 2.2]{PR23} we have $\Aut_{alg}(X)=T\ltimes\UU(X)$; hence we must prove that $\Aut(X)^\circ = \Aut_{alg}(X)$. For this, note that $\Aut(X)$ is generated by $N(T)$ and the maximal unipotent subgroup $\mathcal U(X)$. Indeed, conjugating $T$ by an element $g\in\Aut(X)$, one gets a maximal torus $T'\subseteq \Aut_{alg}(X)$, which is conjugate to $T$ by a unipotent element $u\in \mathcal U(X)$. Hence $u^{-1}g\in N(T)$. Thus it will be enough to show that $N(T)^\circ$ is an algebraic group. Then, since $\mathcal U(X)$ is generated by algebraic elements, so is $\Aut(X)^\circ$.

    In order to prove that $N(T)$ is locally algebraic, consider the Cox realization $X=\overline X /\!/H$. Here $\overline X$ is a trinomial variety, $H$  
    is a quasitorus such that 1) $H$ commutes with $\TT$, and 2) $H^\0$ lies in the canonical torus $\TT$ of~$\overline X$; see~\cite{HW} for this result and~\cite{ADHL} for systematic treatment of Cox rings. By the machinery of the Cox construction, one has a commutative diagram where the first row is exact (cf.~\cite[Theorem 5.1]{AG10}):
    $$
    \begin{tikzcd}
        1 \arrow[r] & H \arrow[r] &N_{\Aut\overline X}(H) \arrow[r, "\pi"] &\Aut(X) \arrow[r] &1\\
        & H^\circ \arrow[u, phantom, sloped, "\subseteq"]\arrow[r, phantom, sloped, "\subseteq"] &\TT \arrow[u, phantom, sloped, "\subseteq"] &T \arrow[u, phantom, sloped, "\subseteq"]
    \end{tikzcd}          
    $$
    We claim that the torus $\pi^{-1}(T)^\circ$ acts on $\overline X$ with complexity one. In fact, since $\overline X$ admits an open subset for which the quotient by $H$ is geometric, and the stabilizer $H_x$ of a generic point $x\in\overline X$ is trivial, it follows that $\dim H=\dim\overline X - \dim X$. Hence 
    $$\dim \pi^{-1}(T) = \dim T +\dim H = \dim\overline X -1.$$
    Moreover, the torus $\pi(\TT)$ is maximal in $\Aut(X)$ since 
    $$\dim\pi(\TT)=\dim \TT-\dim H^\0 = \dim X-1.$$
    But all maximal tori in $\Aut(X)$ are conjugate. Hence (conjugating $T$ if necessary), we may assume that $\pi^{-1}(T)^\circ=\TT$.

    Lemma~\ref{Trinomial-N-Lemma} asserts that $N(\TT)^\circ$ is an algebraic group. One readily verifies that $$\pi^{-1}(N(T)) \subseteq N(\pi^{-1}(T)) \subseteq N(\TT).$$ Hence the normalizer $N(T)$ is a locally algebraic group (being the quotient of a closed subgroup of $N(\TT)$ by the normal subgroup $H$). The assertion follows.
\end{proof}


\end{document}